\documentclass[12pt,a4paper]{article}
\usepackage{cmap}
\usepackage{fullpage}
\usepackage[T2A]{fontenc}
\usepackage[utf8]{inputenc}
\usepackage[english,russian]{babel}
\usepackage{amsmath,amssymb,amsthm}
\usepackage{graphicx}
\usepackage{hyperref}

\hypersetup{
  pdftitle={On the Kepler problem on the Heisenberg group},
  pdfauthor={S.G. Basalaev, S.V. Agapov},
  colorlinks=true,
  linkcolor=blue
}

\newcounter{through}
\numberwithin{through}{section}
\numberwithin{equation}{section}

\newtheorem{proposition}[through]{Proposition}
\newtheorem{theorem}[through]{Theorem}
\newtheorem{corollary}[through]{Corollary}

\theoremstyle{definition}
\newtheorem{remark}[through]{Remark}

\begin{document}
\selectlanguage{english}

\title{Dynamics in the Kepler problem on the Heisenberg group}
\author{Sergey Basalaev\thanks{
The work of the author is supported by the Mathematical Center in Akademgorodok under
the Agreement No. 075-15-2022-282 with the Ministry of Science and Higher Education of the Russian
Federation.}, Sergei Agapov\thanks{
The work of the author was performed according to the
Government research assignment for IM SB RAS, project
FWNF-2022-0004.}}
\date{\today}
\maketitle

\begin{abstract}
We study the nonholonomic motion of a point particle
on the Heisenberg group around the fixed ``sun''
whose potential is given by the fundamental solution
of the sub-Laplacian. We find three independent
first integrals of the system and show that its bounded
trajectories of the system are wound up around certain
surfaces of the fourth order.



\end{abstract}

\section{Introduction}

How would a planet move around the Sun on
the Heisenberg group? While studying the problem
we found that the authors of the paper~\cite{SM}
aim to answer this very question.
However, a known feature of nonholonomic mechanics
(see e.\,g.\ \cite{Bloch}) is that the variatonal problem
(control, geodesics, how to move from A to~B)
and the dynamics problem (how does it move on its own?) are
generally not equivalent. Indeed, for instance, in the
geodesic problem on the Heisenberg group,
to any initial velocity there corresponds a one-parametric
family of geodesics. On the other hand the dynamics is
uniquely determined by an initial position and a velocity
so it can't be \emph{any} geodesic (the actual solutions in
in this case are what is known in nonholonomic geometry as
the ``straightest'' lines).
It seems to us that the problem
actually studied in~\cite{SM} is the variatonal one~---
how to move efficiently on Heisenberg group in the
presence of a gravitational field. Here, we aim to
solve the dynamics problem instead.

We consider the Heisenberg group with the left-invariant
sub-Riemannian metric and a fixed ``sun'' at the origin.
The potential is given by the fundamental solution of the
sub-Laplacian~--- a generalization of the
Laplace--Beltrami operator to the sub-Riemannian manifolds.
Traditionally, to derive the non-holonomic equations of
motion the Lagrange--\hspace{0pt}d'\hspace{0pt}Alembert
principle is used.
In Section~\ref{SecToHamilton} we
remind how the equations of motion can be translated
to the form that uses the intrinsic structure
of nonholonomic distribution. This allows one to use
Hamiltonian language best suited for finding integrals
of the system.
In Section~\ref{SecKepler} we apply
this to study the Kepler problem on the Heisenberg group
and find its first integrals.
In contrast to the 6-dimensional variational problem
which is not Liouville integrable (proved in~\cite{StM}),
the dynamics problem is 5-dimensional and
turns out to have at least three independent first
integrals. This allows us to rather qualitatively
describe the geometry of trajectories of the system.
In particular, a typical trajectory of the system
winds up around the surface of order~4 which we found
explicitly. Due to the nonholonomic constraint the
surface in the Heisenberg group uniquely defines
the trajectory by its starting point. We also
describe a few special trajectories.

In relation to our research we note that the Kepler
problem on the Riemannian manifolds was studied
extensively starting from works of
Lobachevsky \cite{Lobachevsky} in hyperbolic space
and Serret~\cite{Serret} on the sphere. The survey
of related works in the spaces of constant curvature
may be found in~\cite{Diacu}. The aforementioned
paper~\cite{SM} has a few followups~\cite{DS, StM}
all of which seem to address the variational problem.

\section{Motion on sub-Riemannian manifolds}
\label{SecToHamilton}

Here we derive the equations of nonholonomic dynamics
in the generalized Hamiltonian form, simplified for the case
considered. The general form may be found in~\cite{Bloch}.

Consider a mechanical system in $\mathbb{R}^n$
with $k$ ideal functionally independent
nonintegrable constraints linear in
velocities. In Lagrangian coordinates
$q^i$, $\dot{q}^i$ these can be given by
\begin{equation}
\label{EqConstraintsGeneral}
  \sum_{i=1}^n a^j_i(q) \dot{q}^i = 0,
  \quad j = 1, \ldots, k.
\end{equation}
Locally the equations~\eqref{EqConstraintsGeneral}
can be solved to $k$ dependent velocities
and represented in the form
\begin{equation}
\label{EqConstraints}
  \dot{q}^{m+j} = \sum_{i=1}^{m} f^j_i(q) \dot{q}^i,
  \quad j = 1, \ldots, k,
\end{equation}
where $m = n - k$ and the velocities
$\dot{q}^1, \ldots, \dot{q}^m$ are assumed
to be independent.

Recall that for a nonholonomic system with the
Lagrangian $L(q, \dot{q}, t)$ and the constraints~\eqref{EqConstraintsGeneral}
the equations
of motion are derived using the Lagrange--d'Alembert
principle (see, e.\,g. \cite{Bloch})
\begin{equation}
\label{EqDalembertPrinciple}
  \frac{d}{dt} \frac{\partial L}{\partial \dot{q}^i}
  - \frac{\partial L}{\partial q^i}
  = \sum_{j=1}^k \lambda_j a^j_i,
  \quad
  i = 1, \ldots, n,
\end{equation}
where the Lagrange multipliers $\lambda_j$ are
determined in such a way that the trajectory satisfies
constraints~\eqref{EqConstraintsGeneral}.

It may be useful, especially for problems
with the constraints of
form~\eqref{EqConstraints},
instead of the Lagrangian
coordinates use the ones
in the distribution of admissible velocities.
Consider the vector fields
\begin{equation}
\label{EqVectorFields}
  X_i(q) = {\partial q_i}
      + \sum_{j=1}^k f^j_i(q)
      {\partial q_{m+j}},
  \quad i = 1, \ldots, m.
\end{equation}
Then the velocity $\dot{q}$ satisfies the
constraints~\eqref{EqConstraints} iff
$\dot{q} = \sum\limits_{i=1}^m \dot{q}^i X_i$.
Introduce the momentum 1-form
\begin{equation}
\label{EqMomentumForm}
  P
  = \sum_{i=1}^n \frac{\partial L}{\partial \dot{q}^i}
    dq^i.
\end{equation}
Then we can describe the dynamics by the following
generalization of Euler--Lagrange equations
(in what follows we denote the action of
1-form $\tau$ on the vector field $X$
as $\tau \langle X \rangle$).

\begin{proposition}
The dynamical motion in the system with the Lagrangian
$L(q, \dot{q}, t)$ and the constraints~\eqref{EqConstraints} is described
by the system of equations
\begin{align}
\label{EqEulerLagrange}
  & \frac{d}{dt} P \langle X_i \rangle
  = X_i L,
  & i = 1, \ldots, m,
  \\
  & \dot{q}^{m+j} = \sum_{i=1}^{m} f^j_i(q) \dot{q}^i,
  & j = 1, \ldots, k.
  \notag
\end{align}
\end{proposition}

\begin{proof}
Introduce 1-forms of our constraints
\begin{equation}
\label{EqConstraintsForms}
  \tau^j(q) = dq^{m+j} - \sum_{i=1}^m f^j_i(q) \, dq^i,
  \qquad
  j = 1, \ldots, k.
\end{equation}
Then $\tau^j \langle X_i \rangle = 0$
for $i = 1, \ldots, m$ and
$\tau^j \langle \partial q_{m+l} \rangle = \delta^j_l$ for $l = 1, \ldots, k$.
The Lagrange--d'Alembert
equations~\eqref{EqDalembertPrinciple}
can be rewritten in our terms as
\[
  \frac{d}{dt} P
  \langle {\partial q_i} \rangle
  - d L
  \langle {\partial q_i} \rangle
  = \sum_{j=1}^k \lambda_j \tau^j
  \langle {\partial q_i} \rangle,
  \qquad i = 1, \ldots, n.
\]
Then, since the expression is linear w.\,r.\,t. the
term in the angle brackets
\[
  \frac{d}{dt} P
  \langle X_i \rangle
  - d L \langle X_i \rangle
  = \sum_{j=1}^k \lambda_j \tau^j
  \langle X_i \rangle = 0,
  \qquad i = 1, \ldots, m.
\]
The sufficiency of the
equations~\eqref{EqConstraints},
\eqref{EqEulerLagrange} follows from
the fact that we can recover Lagrange--d'Alembert
equations from them. Indeed, let
\[
  \lambda_j =
  \frac{d}{dt} P
  \langle {\partial q_{m+j}} \rangle
  - d L
  \langle {\partial q_{m+j}} \rangle
  = \frac{d}{dt} \frac{\partial L}{\partial \dot{q}^{m+j}} -
  \frac{\partial L}{\partial q^{m+j}},
  \quad
  j = 1, \ldots, k.
\]
This gives us the equations~\eqref{EqDalembertPrinciple}
for $i = m+1, \ldots, n$. Then, since
$\partial q_i = X_i - \sum\limits_{j=1}^k f^j_i(q) {\partial q_{m+j}}$
for $i = 1, \ldots, m$ we have
\[
  \frac{d}{dt} \frac{\partial L}{\partial \dot{q}^{i}} -
  \frac{\partial L}{\partial q^{i}}
  = \frac{d}{dt} P
  \langle X_i \rangle
  - d L \langle X_i \rangle
  - \sum_{j=1}^k f^j_i
  \Big( \frac{d}{dt} \frac{\partial L}{\partial \dot{q}^{m+j}} -
  \frac{\partial L}{\partial q^{m+j}} \Big)
  = - \sum_{j=1}^k \lambda_j f^j_i.
\]
These are the
equations~\eqref{EqDalembertPrinciple}
for $i = 1, \ldots, m$.
Thus, the system of
equations~\eqref{EqConstraints},
\eqref{EqEulerLagrange} is equivalent
to the one of~\eqref{EqConstraints},
\eqref{EqDalembertPrinciple}.
\end{proof}

Let $\mathcal{D}$ be the distribution
spanned by $X_1, \ldots, X_m$,
i.\,e.
$\mathcal{D}_q = \mathrm{span} \,
\{ X_1(q), \ldots, X_m(q) \}$.
One thing to note is that for deriving
equations~\eqref{EqEulerLagrange} for a particular
system it is enough to know the Lagrangian only
on $\mathcal{D}$, not on the whole $T \mathbb{R}^n$,
which allows us to immerse the problem
in the sub-Riemannian setting.

Recall that the (regular) sub-Riemannian structure
on a smooth manifold $M$ is given by the
constant rank distribution $\mathcal{D} \subset TM$
(i.\,e. $\mathcal{D}_x \subset T_x M$ is a subspace
and $\dim \mathcal{D}_x$ is independent of~$x$)
and the sub-Riemannian metric tensor
$\langle \cdot , \cdot \rangle$ on $\mathcal{D}$,
i.\,e. $\langle \cdot , \cdot \rangle_x$ is
a scalar product on $\mathcal{D}_x$.

Let us reformulate the problem in
Hamiltonian terms.
The \emph{energy} $E(q, \dot{q}, t)$ of the system is
defined as usual:
\[
  E = P \langle \dot{q} \rangle - L
\]
and satisfies $\frac{d}{dt} E = -\frac{\partial}{\partial t} L$ on trajectories of the system.
Note, that the dual basis to the
one of vector fields
$X_1, \ldots, X_m, \partial q_{m+1},
\ldots, \partial q_{m+k}$ consists of 1-forms
\[
  dq^1, \ldots, dq^m, \tau^{1}, \ldots, \tau^k.
\]
In particular, $dq^1, \ldots, dq^m$ form the basis
of $\mathcal{D}^*$. Introduce the momenta
$p = \sum\limits_{i=1}^m p_i \, dq^i$ on $\mathcal{D}^*$,
i.\,e.\ $p \langle X_i \rangle = p_i$,
$i = 1, \ldots, m$.
Assuming that
$\dot{q}(p, q, t) \in \mathcal{D}_q$ can be determined
uniquely from the equation
$p \langle \dot{q} \rangle = P \langle \dot{q} \rangle$ let us define the \emph{generalized Hamiltonian}
$H(p, q, t)$ on $\mathcal{D}^*$ as
\[
  H(p, q, t) = p \langle \dot{q} \rangle - L(q, \dot{q}, t)
  = \sum_{i=1}^m p_i \dot{q}^i - L(q, \dot{q}, t).
\]
For this assumption to take place
it is sufficient to require that the restriction of
the quadratic form
$\frac{\partial^2 L}{\partial \dot{q}^i \partial \dot{q}^j} dq^i dq^j$ on $\mathcal{D}$ is positive
definite.

Reformulating the
equations~\eqref{EqEulerLagrange} in terms of $H$
one obtains

\begin{proposition}
\label{PropHamiltonian}
The dynamical motion in the nonholonomic system with
the constraints~\eqref{EqConstraints}
and the generalized Hamiltonian
$H(p, q, t)$ on $\mathcal{D}^*$ is described
by the system of equations
\begin{align}
\label{EqHamilton}
  & \dot{q}^i = \frac{\partial H}{\partial p_i},
  \quad \dot{p}_i = - X_i H,
  & i = 1, \ldots, m,
  \\
  & \dot{q}^{m+j} = \sum_{i=1}^{m} f^j_i(q) \frac{\partial H}{\partial p_i},
  & j = 1, \ldots, k.
  \notag
\end{align}
\end{proposition}

\begin{proof}
Observe that since the bases
$X_1, \ldots, X_m, \partial q^{m+1}, \ldots, \partial^{m+k}$,
introduced in~\eqref{EqVectorFields},
and $d\dot{q}^1, \ldots, d\dot{q}^m, \tau^1, \ldots, \tau^k$,
introduced in~\eqref{EqConstraintsForms},
are dual, for a smooth function $f(q)$ we have
\[
  df
  = \sum_{i=1}^m df \langle X_i \rangle \, dq^i
  + \sum_{j=1}^k df \langle \partial q^{m+j} \rangle \, \tau^j
  = \sum_{i=1}^m X_i f \, dq^i
  + \sum_{j=1}^k \frac{\partial f}{\partial q^{m+j}} \, \tau^j.
\]
Then
\begin{multline*}
  dL = \sum_{i=1}^m X_i L \, dq^i
  + \sum_{j=1}^k \frac{\partial L}{\partial q^{m+j}} \, \tau^j
  + \sum_{i=1}^n \frac{\partial L}{\partial \dot{q}^i} d\dot{q}^i
  + \frac{\partial L}{\partial t} dt \\
  = \sum_{i=1}^m X_i L \, dq^i
  + \sum_{j=1}^k \frac{\partial L}{\partial q^{m+j}} \, \tau^j
  + \sum_{i=1}^n P \langle \partial q^i \rangle d\dot{q}^i
  + \frac{\partial L}{\partial t} dt.
\end{multline*}
Further, restricting $L$ on $\mathcal{D} \times \mathbb{R}$ with coordinates $q^1, \ldots, q^n$,
$\dot{q}^1, \ldots, \dot{q}^m$, $t$,
i.\,e.\ setting
$\dot{q} = \dot{q}^1 X_1 + \ldots + \dot{q}^m X_m$,
we obtain
\[
  dL|_{T\mathcal{D}}
  = \sum_{i=1}^m X_i L \, dq^i
  + \sum_{j=1}^k \frac{\partial L}{\partial q^{m+j}} \, \tau^j
  + \sum_{i=1}^m P \langle X_i \rangle d\dot{q}^i
  + \frac{\partial L}{\partial t} dt.
\]
Now, for $H = p \langle \dot{q} \rangle - L(q, \dot{q}, t)|_\mathcal{D}$ with $\dot{q} = \dot{q}(p, q, t)$ we have
\[
  dH = \sum_{i=1}^m
  \big( p_i - P \langle X_i \rangle \big) \, d\dot{q}^i
  + \sum_{i=1}^m \dot{q}^i \, dp_i
  - \sum_{i=1}^m X_i L \, dq^i
  - \sum_{j=1}^k \frac{\partial L}{\partial q^{m+j}} \, \tau^j
  - \frac{\partial L}{\partial t} dt.
\]
If $\dot{q}$ satisfies
$p \langle \dot{q} \rangle = P \langle \dot{q} \rangle$
then the first sum vanishes. Finally, since
for the function $H(p, q, t)$
\[
  dH
  = \sum_{i=1}^m \frac{\partial H}{\partial p_i}
  dp_i
  + \sum_{i=1}^m X_i H \, dq^i
  + \sum_{j=1}^k \frac{\partial H}{\partial q^{m+j}} \tau^j
  + \frac{\partial H}{\partial t} dt,
\]
the equations~\eqref{EqHamilton} follow
from~\eqref{EqEulerLagrange}.
\end{proof}

Observe that for time independent system $H$
is a first integral of~\eqref{EqHamilton},
\eqref{EqConstraints}.
We can define \emph{almost Poisson bracket}
(see e.\,g. \cite[Section 3.1]{Bloch})
on $\mathcal{D}^*$ as
\[
  \{ F, H \}
  = \sum_{i=1}^m
  \Big( \frac{\partial H}{\partial p_i} X_i F
  - \frac{\partial F}{\partial p_i} X_i H \Big).
\]
It retains all the properties of Poisson bracket
but the Jacobi identity. Nevertheless,
since $\dot{F} = \{ F, H \}$
it follows that $F$ is an integral
of~\eqref{EqHamilton},
\eqref{EqConstraints} iff $\{ F, H \} \equiv 0$.

\section{Motion in a potential field on the Heisenberg group}
\label{SecKepler}

Recall that the Heisenberg group
$\mathbb{H}^1 = (\mathbb{R}^3, \cdot, \delta_\lambda)$
is a homogeneous group with the group operation
\[
  (x, y, z)
  \cdot
  (x', y', z')
  =
  \Big(
  x + x', y + y', z + z' + \frac{xy' - x'y}{2}
  \Big),
\]
and the one-parametric family of anisotropic
dilatations
\[
  \delta_\lambda(x, y, z)
  = (\lambda x, \lambda y, \lambda^2 z),
  \quad \lambda > 0.
\]
Its Lie algebra $\mathfrak{h}^1$ of
left-invariant vector fields has the basis
\[
  X = \partial_x - \frac{y}{2} \partial_z,
  \quad
  Y = \partial_y + \frac{x}{2} \partial_z,
  \quad
  Z = [X, Y] = \partial_z.
\]
The dual basis of left-invariant 1-forms is
\[
  dx, \quad
  dy, \quad
  \tau = dz + \frac{y \, dx - x \, dy}{2}.
\]
The \emph{horizontal distribution}
$\mathcal{D} =
\mathop{\mathrm{span}} \{ X, Y \} \subset T \mathbb{H}^1$ is totally nonholonomic.
The form $\tau$ is its annihilator.
The sub-Riemannian structure on $\mathbb{H}^1$
is given by the quadratic form
$\langle \cdot, \cdot \rangle$ on $\mathcal{D}$.
We choose the one such that $X, Y$ form the
orthonormal basis:
\[
  ds^2 = dx^2 + dy^2.
\]
While this quadratic form is degenerate on
$T \mathbb{H}^1$ it is positive definite on
$\mathcal{D}$.
For the mechanical motion with the kinetic energy
$T = \frac 12 ds^2\langle \dot{q} \rangle = \frac 12
(\dot{x}^2 + \dot{y}^2)$ and the potential energy
$U = U(x, y, z)$ one has, as usual, the Lagrangian
$L = T - U$. By Proposition~\ref{PropHamiltonian}
we can translate equations to the Hamiltonian form
where the Hamiltonian
$H$ on $\mathcal{D}^*$ takes the form
$H = 2 T - L = T + U$, i.\,e.
\[
  H(x, y, z, p_X, p_Y)
  = \frac{p_X^2 + p_Y^2}{2} + U(x, y, z).
\]
We are interested in the gravitational potential
which in $\mathbb{R}^n$ is given by a fundamental
solution of the Laplacian.
The analogue of Laplace--Beltrami operator
on the Heisenberg group is the operator
$\Delta_H= X^2 + Y^2$. Its fundamental
solution\footnote{In the cited paper~\cite{SM} the
potential $U$ has the term $\frac{z^2}{16}$ instead of
$16z^2$. One can check that $16$ is the correct
coefficient since only in this case
$\Delta_H U = 0$ away from the origin.}
is found in~\cite{F}:
\[
  U = - \frac{k}{\rho^2},
  \quad \text{where }
  \rho(x, y, z) = ((x^2 + y^2)^2 + 16 \* z^2)^{\frac 14}
\]
and $k > 0$ is some constant.
Since both the distribution and the potential have
a rotational symmetry around $Oz$ it is natural to
make the cylindrical coordinate change
$x = r \cos \theta, y = r \sin \theta$. The basis of
$\mathcal{D}$
may be given by vector fields
\begin{align*}
  R &= \phantom{-r} \cos \theta \, X
       + \phantom{r} \sin \theta \, Y
    = \partial_r,
  \\
  S &= - r \sin \theta \, X + r \cos \theta \, Y
    = \partial_\theta
    + \tfrac{r^2}{2} \partial_z.
\end{align*}
Duals to the basis $R, S, \partial_z$ are
$dr, d\theta, \tau = dz - \frac{r^2}{2} d\theta$
and for the momenta we have
\[
  p_X \, dx + p_Y \, dy
  = (p_X \cos \theta + p_Y \sin \theta) dr
  + r (p_Y \, \cos \theta - p_X \, \sin \theta) d\theta
  = p_R \, dr + p_S \, d\theta.
\]
It follows that $T = \frac{p_R^2 + p_S^2 / r^2}{2}$ and the Hamiltonian becomes
\[
    H(r, \theta, z, p_R, p_S)
    = \frac{p_R^2 + \frac{1}{r^2} p_S^2}{2}
    - \frac{k}{(r^4 + 16 z^2)^{\frac 12}}.
\]
Since constraints in the new coordinates still have
the form~\eqref{EqConstraints} we may apply
Proposition~\ref{PropHamiltonian} to derive
the equations of motion:
\begin{align}
  \dot{r} &= \frac{\partial H}{\partial p_R} = p_R, &
  \dot{p}_R &= -R H = \frac{p_S^2}{r^3}
    - \frac{2 k r^3}{(r^4 + 16 z^2)^{\frac 32}},
  \notag
  \\
  \dot{\theta} &= \frac{\partial H}{\partial p_S}
  = \frac{p_S}{r^2}, &
  \dot{p}_S &= -S H
  = -\frac{8 k r^2 z}{(r^4 + 16 z^2)^{\frac 32}},
  \notag
  \\
  \dot{z} &= \frac{r^2}{2} \frac{\partial H}{\partial p_S}
  = \frac{p_S}{2}.
  \label{EqCylindrical}
\end{align}

\begin{theorem}
\label{ThBounded}
If $H < 0$ the solutions of~\eqref{EqCylindrical} are bounded with
$\sqrt{r^4 + 16 z^2} \le \frac{k}{|H|}$.
\end{theorem}

\begin{proof}
This easily follows from
the inequality
$H + k / \sqrt{r^4 + 16 z^2} = \frac{p_R^2 + \frac{1}{r^2} p_S^2}{2} \ge 0$.
\end{proof}

In what follows we search for the additional
first integrals of the system.

\begin{proposition}
\label{PropFirstOrder}
The system~\eqref{EqCylindrical} does not admit any
linear in momenta first integrals.
\end{proposition}

This statement can be checked by straightforward
calculations. We skip the details.
However, it turns out that there are a
few quadratic integrals in addition to the
Hamiltonian~$H$.

\begin{theorem}
\label{PropSecondOrder}
The system~\eqref{EqCylindrical}
admits quadratic in momenta first integrals
\begin{align*}
  F_1 &=
  \phantom{-} \Big( p_Rp_Sr-2p_R^2z+\frac{2p_S^2z}{r^2} \Big) \cos (2\theta)
  + \Big( \frac{4 p_Rp_Sz}{r} -p_S^2 +\frac{kr^2}{\sqrt{r^4+16z^2}} \Big) \sin (2\theta),
\\
  F_2 &=
  - \Big( p_Rp_Sr-2p_R^2z+\frac{2p_S^2z}{r^2} \Big) \sin (2\theta)
  + \Big( \frac{4 p_Rp_Sz}{r}-p_S^2+\frac{kr^2}{\sqrt{r^4+16z^2}} \Big) \cos (2\theta),
\\
  F_3 &=
  (2 z p_R - r p_S)^2
  + 4 z^2 \Big( \frac{p_S^2}{r^2}
  + \frac{2k}{\sqrt{r^4+16z^2}} \Big).
\end{align*}
Any three of $H, F_1, F_2, F_3$ are
functionally independent a.\,e.\ wherein all of them
satisfy the relation
\begin{equation}
\label{EqIntegralRelation}
  F_1^2 + F_2^2 = 2 H F_3 + k^2.
\end{equation}
\end{theorem}

This theorem can be verified by straightforward
calculations.
The method we used to construct these integrals is described in Appendix~\ref{AppendixIntegrals}.

Knowing three independent first integrals allows us
to derive the equation of the surface (in coordinates
$(r, \theta, z)$) in which the trajectories lie. To do
that we introduce two more conserved quantities
$J \ge 0$ and in the case $J > 0$ also $\theta_0 \in [0, \pi)$ such that
\[
  F_1 = J \sin(2\theta_0),
  \quad
  F_2 = J \cos(2\theta_0).
\]
As main parameters we choose $H, F_3$ that do not
depend on the angle and the phase offset $\theta_0$
that captures the rotational symmetry of the problem.
Note, that from the definition $F_3 \ge 0$ since $k > 0$,
and $2 H F_3 + k^2 = J^2 \ge 0$. Therefore we have one general case
and two cases that might require special handling:

\begin{itemize}

\item
The general case $F_3 > 0$, $J > 0$.
Then $H > - \frac{k^2}{2 F_3}$ and $\theta_0 \in [0, \pi)$
is defined.

\item
The minimum energy case $F_3 > 0$, $J = 0$.
Then $H = H_{\min} = - \frac{k^2}{2 F_3}$,
$\theta_0$ is undefined.

\item
The degenerate case $F_3 = 0$.
In this case $J = k$, $\theta_0 \in [0, \pi)$ is defined
and $H$ is unbounded.

\end{itemize}

\begin{theorem}
\label{PropSurfaces}
All trajectories of the system~\eqref{EqCylindrical}
with the fixed values of the first integrals
$H, F_3, \theta_0$ lie on the surface which in the general
case $F_3 > 0$, $J > 0$ satisfies the equation
\begin{equation}
\label{EqSurface}
  F_3
  = 8 z^2 H
  + k \sqrt{r^4 + 16 z^2}
  - \sqrt{k^2 + 2H F_3} \, r^2
    \cos \big( 2 (\theta - \theta_0) \big).
\end{equation}

In the minimum energy case $F_3 > 0$, $J = 0$ the surface
becomes an ellipsoid of revolution
\begin{equation}
\label{EqSurfaceMinH}
  4 k^2 z^2 + k F_3 r^2 = F_3^2.
\end{equation}

In the degenerate case $F_3 = 0$ the surface degenerates to
the straight horizontal line passing through the origin
\begin{equation}
\label{EqSurfaceMinF3}
z = 0, \quad \theta = \theta_0 \mod \pi.
\end{equation}
\end{theorem}

\begin{proof}
Observe, that we can write $F_3$ as
\begin{align*}
  F_3 =
  8z^2 H
  + k \sqrt{r^4 + 16 z^2}
  - r^2 \Big(- p_S^2 + \frac{4 p_R p_S z}{r}
  + \frac{k r^2}{\sqrt{r^4 + 16 z^2}} \Big).
\end{align*}
From the expressions of $F_1$ and $F_2$ we have
\[
  \Big(- p_S^2 + \frac{4 p_R p_S z}{r}
  + \frac{k r^2}{\sqrt{r^4 + 16 z^2}} \Big)
  = F_1 \sin(2\theta) + F_2 \cos(2\theta).
\]
Therefore,
\begin{equation}
\label{EqSurfProof1}
  F_3 =
  8z^2 H
  + k \sqrt{r^4 + 16 z^2}
  - r^2 \big( F_1 \sin(2\theta) + F_2 \cos(2\theta) \big).
\end{equation}
Let $F_3 > 0$ and $J > 0$.
We have $F_1 \sin(2\theta) + F_2 \cos(2\theta) =
J \cos(2(\theta - \theta_0))$ and~\eqref{EqSurface}
follows.

Now, let $F_3 > 0$ and $J = 0$.
In this case $F_1 = F_2 = 0$
and the last term in~\eqref{EqSurfProof1} vanishes.
Since $H = -\frac{k^2}{2F_3}$ in this
case,~\eqref{EqSurfProof1} becomes
\[
  F_3^2 +
  4 k^2 z^2
  = k F_3 \sqrt{r^4 + 16 z^2}.
\]
Squaring and simplifying it we obtain
\begin{equation}
\label{EqSurfProof2}
  (F_3^2 - 4 k^2 z^2)^2 = k^2 F_3^2 r^4.
\end{equation}
By Theorem~\ref{ThBounded} for the trajectories of the system
$\sqrt{r^4 + 16 z^2} \le \frac{k}{|H|} = \frac{2 F_3}{k}$.
Therefore,
\[
  F_3 \ge \tfrac{k}{2} \sqrt{r^4 + 16 z^2}
  \ge \tfrac{k}{2} \sqrt{16 z^2} = 2 k z.
\]
Now~\eqref{EqSurfaceMinH} follows if we
take the square root of~\eqref{EqSurfProof2}.

Lastly, $F_3 = 0$ implies $z = 0$ and $J = k > 0$.
The restriction of~\eqref{EqSurfProof1} on $z = 0$ becomes
\[
  0 =
  k r^2
  - k r^2 \cos (2(\theta - \theta_0)).
\]
This gives us either $r = 0$ (the origin) or
$\theta = \theta_0 \mod \pi$.
\end{proof}

\begin{remark}
Solving the equation~\eqref{EqSurface} for the square root
$\sqrt{r^4 + 16 z^2}$ and then squaring it we obtain
the following equation
\[
  k^2 (r^4 + 16 z^2)
  = \big( F_3
  - 8 z^2 H
  + \sqrt{k^2 + 2H F_3} \, r^2
    \cos \big( 2 (\theta - \theta_0) \big) \big)^2,
\]
or in Cartesian coordinates
\[
  k^2 ((x^2 + y^2)^2 + 16 z^2)
  = \big( F_3
  - 8 z^2 H
  + \sqrt{k^2 + 2H F_3} \,
    (\cos(2\theta_0) (x^2 - y^2) + 2 \sin(2\theta_0) xy)
    \big)^2.
\]
We see that this is an equation of the fourth order.
However, its solution is a branched surface and only
one of its branches is the solution to the original
equation, i.\,e.\ the equation is quadratic in $z^2$ but
only one of its two roots solves~\eqref{EqSurface}.
\end{remark}

Examples of the surfaces corresponding
to the cases $H = 0$ and $H < 0$
may be seen on Fig.~\ref{FigSurfaces}.
Next we note a few properties of the surfaces obtained.

\begin{figure}[ht]

\includegraphics[width=0.5\textwidth]{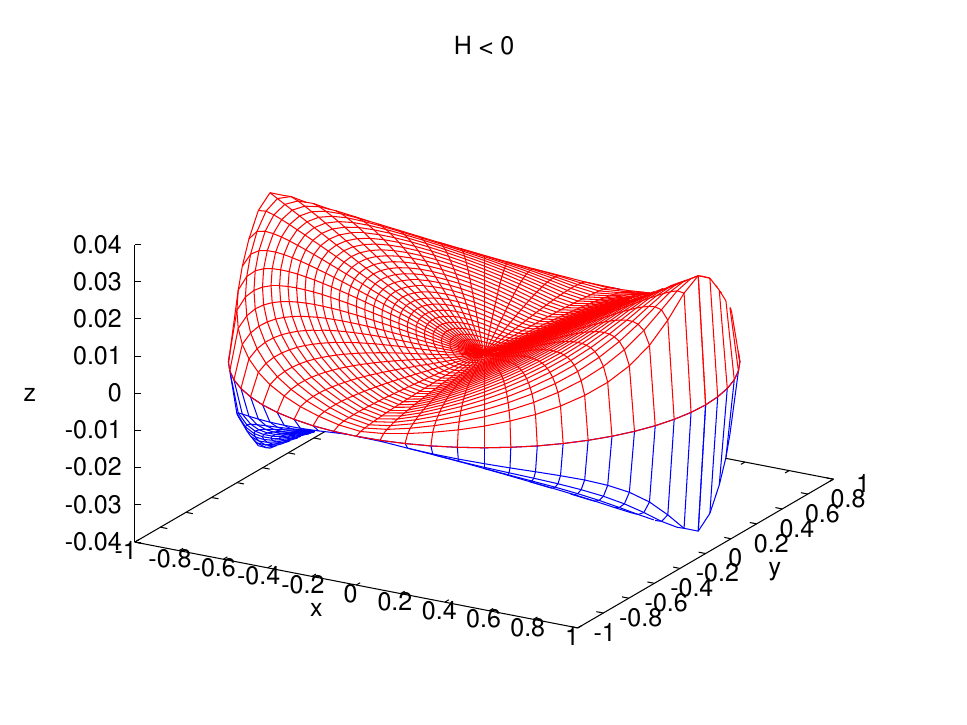}
\includegraphics[width=0.5\textwidth]{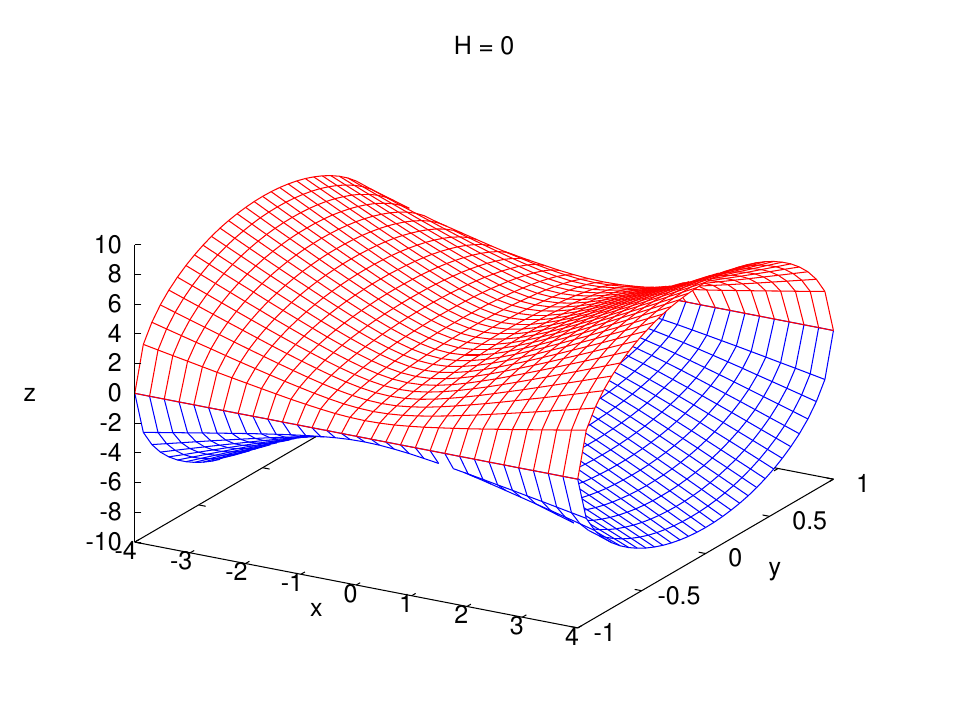}
\includegraphics[width=0.5\textwidth]{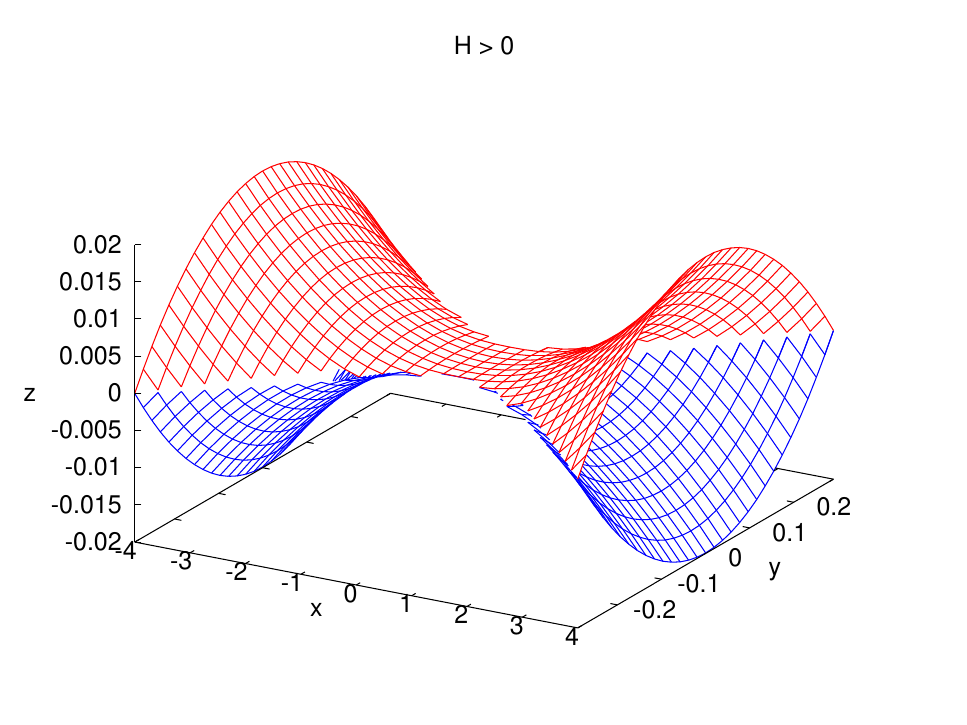}

\caption{Surfaces corresponding to the cases $H < 0$, $H = 0$ and $H > 0$ respectively.}
\label{FigSurfaces}
\end{figure}

\begin{corollary}
\label{ThSurfaceGeometry}
In the non-degenerate case $F_3 > 0$
the surfaces described in Theorem~\ref{PropSurfaces}
have the following properties.
\begin{enumerate}

\item
The surface is topologically

\begin{enumerate}

\item
a sphere in the case $H < 0$;

\item
a cylinder in the case $H \ge 0$.

\end{enumerate}

\item
The surface has reflection symmetry in the plane $z = 0$
and

\begin{enumerate}

\item
in the case $J > 0$
it has two more planes of symmetry
$\theta = \theta_0$ and $\theta = \theta_0 + \frac{\pi}{2}$;

\item
in the case $J = 0$ it is the surface of
revolution around $r = 0$.

\end{enumerate}

\item
The trace of the surface on the plane $z = 0$
is a quadratic curve:

\begin{enumerate}
\item
in the case $H < 0$ it is an ellipse with the semiaxes
$\sqrt{\frac{F_3}{k-J}}$ and $\sqrt{\frac{F_3}{k+J}}$;

\item
in the case $H = 0$ it is two parallel lines at the distance
$\sqrt{\frac{F_3}{k}}$ from the origin;

\item
in the case $H > 0$ it is a hyperbola with the semiaxis
$\sqrt{\frac{F_3}{k+J}}$.
\end{enumerate}

\end{enumerate}
\end{corollary}

The properties are straightforward
and easy to check.

A smooth surface $S \subset \mathbb{H}^1$
is transversal to the horizontal distribution
$\mathcal{D}$ at almost all points, i.\,e.
$T_x S \cap \mathcal{D}_x$ is one-dimensional
for a.\,e.\ $x \in S$. Therefore, the trajectory of
the system is rather uniquely defined by a starting
point on a surface, with the only possible exception
being when the solution arrives at the
point tangent to $\mathcal{D}$ with zero velocity.
An example of a bounded trajectory ($H < 0$)
which we belive to be a typical one is presented in Fig.~\ref{FigTrajectory}. Next we describe
special solutions corresponding
to the degenerate cases.

\begin{figure}[ht]

\includegraphics[width=0.5\textwidth]{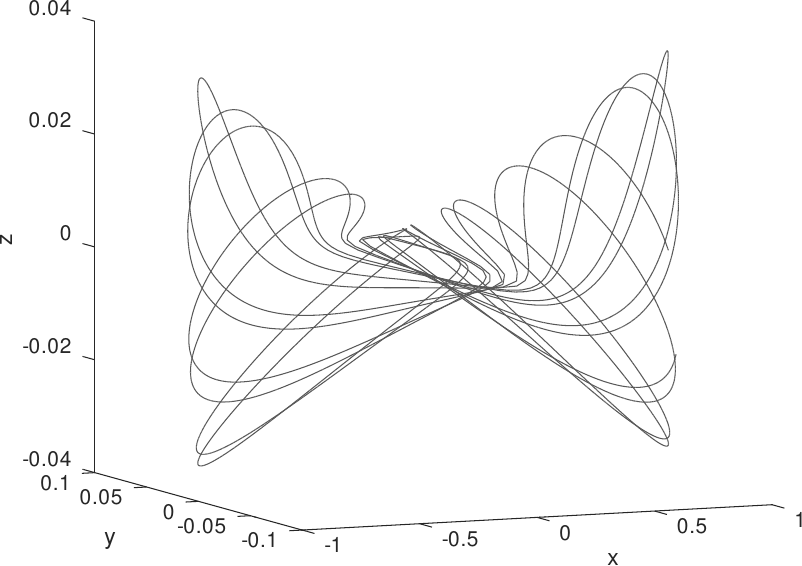}
\includegraphics[width=0.5\textwidth]{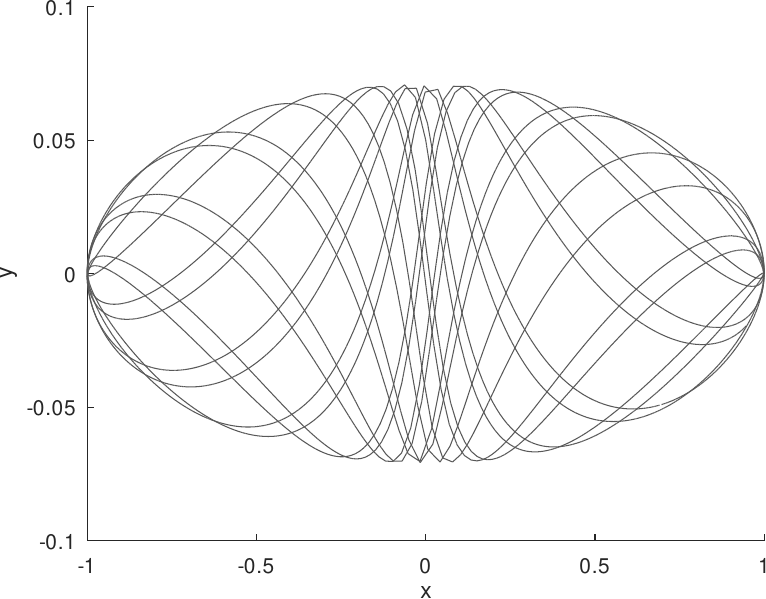}

\caption{A trajectory and its projection on the plane $Oxy$
with $k = 1$ and initial values
$x(0) = 1$, $y(0)= z(0) = p_X(0) =0$,
$p_Y(0) = \frac{1}{10}$.}
\label{FigTrajectory}
\end{figure}

\begin{theorem}
The only trajectories of~\eqref{EqCylindrical} passing
through the origin are straight lines in the plane
$z = 0$. In this case $\theta = const$ and
$r(t)$ satisfies the equation
$\frac{\dot{r}^2}{2} = H + \frac{k}{r^2}$.
These solutions correspond to the degenerate case
$F_3 = 0$.
\end{theorem}

\begin{proof}
From Corollary~\ref{ThSurfaceGeometry} the
trajectories may pass through the origin only in the
case $F_3 = 0$, i.\,e.\ only if $z \equiv 0$
and $p_S \equiv 0$.
Then $\dot{\theta} = 0$ and the Hamiltonian
becomes $H = \frac{\dot{r}^2}{2} - \frac{k}{r^2}$.
\end{proof}

Therefore, the conserved quantity $F_3$ serves as
a kind of angular/vertical momentum. Two more special
solutions appear in the minimal energy case $J = 0$.

\begin{theorem}
\label{ThStationary}
The only stationary solutions of~\eqref{EqCylindrical}
are points on $Oz$:
\[
  r = 0,
  \quad
  z = \pm \frac{k}{4H}.
\]
These solutions correspond to the
minimal energy case $J = 0$.
\end{theorem}

\begin{proof}
Indeed, outside of the axis $r = 0$ the stationary solution
must satisfy $p_S \equiv 0$. But in this case
$\dot{p}_R$ is negative and the solution
is non-stationary. For the stationary solution on $Oz$
we have $H = - \frac{k}{\sqrt{16 z^2}}$
and $F_3 = \frac{ 8 k z^2}{\sqrt{16z^2}}$.
Therefore $J^2 = k^2 + 2 H F_3 = 0$.
\end{proof}

\begin{theorem}
Let $J = 0$ and $z_0 = \frac{k}{4|H|}$.
The trajectories of non-stationary solutions
to~\eqref{EqCylindrical} are the curves monotone
in $z$ and $\theta$ such that being parameterized by $z$
they have the form
\[
  r(z)
  = \Big( 2 \frac{z_0^2 - z^2}{z_0} \Big)^{\frac 12},
  \qquad
  \theta(z) = \frac 12 \log \frac{z_0 + z}{z_0 - z} + \theta(0),
  \qquad |z| \le z_0.
\]
These solutions connect the stationary points
$(0,0,\pm z_0)$ and take infinite time to approach them,
i.\,e. $t(z) \to \pm \infty$ as $z \to \pm z_0$.
\end{theorem}

\begin{proof}
Note that $F_3 \ge 0$. Therefore,
$J = \sqrt{k^2 + 2HF_3} = 0$ implies $H < 0$.
Hence, $H = - \frac{k}{4 z_0}$ and
$F_3 = -\frac{k^2}{2 H} = 2 k z_0$.
The surface~\eqref{EqSurfaceMinH} in this case
is an ellipsoid of revolution
\[
  \frac{r^2}{2 z_0} + \frac{z^2}{z_0^2} = 1.
\]
From this equation we find the dependence $r(z)$.
Next, from $\frac{d\theta}{dz}
= \frac{\dot{\theta}}{\dot{z}} = \frac{2}{r^2}$
we also find the expression of $\theta(z)$
in a closed form.
This gives us a family of curves described in the
statement of the theorem. Take any such curve.
From the expression of $H$ we have
\begin{equation}
\label{EqDotZ}
  \frac{\dot{r}^2 + \frac{4 \dot{z}^2}{r^2}}{2}
  = \frac{p_R^2 + \frac{p_S^2}{r^2}}{2}
  = H + \frac{k}{\sqrt{r^4 + 16 z^2}}
  = \frac{k (z_0^2 - z^2)}{4 z_0 (z_0^2 + z^2)}
  > 0
\end{equation}
for all points except the poles. Therefore, the velocity
along the curve is nonzero except on the endpoints.
Hence the solution restricted to a curve is monotone in
$z$.
From the equation of the surface we have
\[
  0 = r \dot{r} + \frac{2 z \dot{z}}{z_0}.
\]
This together with~\eqref{EqDotZ} yields the equation
on $z$:
\[
  \dot{z}^2
  = \frac{k (z_0^2 - z^2)^2}{4 (z^2 + z_0^2)^2}.
\]
Choosing the solution increasing in $z$ we obtain
\[
  {t(z_1)
  = \int_0^{z_1} \frac{dz}{\dot{z}}
  = \frac{2}{\sqrt{k}} \int_0^{z_1}
    \frac{z_0^2 + z^2}{z_0^2 - z^2} \, dz}
\]
which diverges as $z_1 \to \pm z_0$.
The theorem is proved.
\end{proof}

\section{Conclusion}

In conclusion we see that the variational problem and
the dynamics problem on the Heisenberg group are vastly
different. While the first is Hamiltonian but
non-integrable in Liouville sense, the second one,
being non-Hamiltonian has at least three first integrals.
Both problems are interesting but provide a different
insight into the nonholonomic world.

\textbf{Author Contributions}:
All results except Proposition~\ref{PropFirstOrder},
Theorem~\ref{PropSecondOrder} and Appendix belong to
S.~Basalaev. Proposition~\ref{PropFirstOrder},
Theorem~\ref{PropSecondOrder} and Appendix are
contributed by S.~Agapov.

\textbf{Acknowledgements}:
The first author he is very grateful to second author for
deriving the first integrals (Proposition~\ref{PropFirstOrder},
Theorem~\ref{PropSecondOrder} and Appendix) which significantly progressed the research
and for overall critical comments.

The images were prepared using GNUPlot and Maxima
free software.

\appendix

\section{Derivation of quadratic first integrals}
\label{AppendixIntegrals}

By definition any first integral $F$
of~\eqref{EqCylindrical} must satisfy
the following relation:
\begin{equation}
\label{EqForIntegral}
  \frac{dF}{dt}
  = \frac{\partial F}{\partial r} p_R
  + \frac{\partial F}{\partial \theta} \frac{p_S}{r^2}
  + \frac{\partial F}{\partial z} \frac{p_S}{2}
  + \frac{\partial F}{\partial p_R} \left( \frac{p_S^2}{r^3}-\frac{2kr^3}{(r^4+16z^2)^{3/2}} \right) - \frac{\partial F}{\partial p_S} \frac{8kr^2z}{(r^4+16z^2)^{3/2}} = 0.
\end{equation}
It is quite natural to search for the first integrals of~\eqref{EqCylindrical} having the form of non-homogeneous polynomials in momenta.

We shall search for the quadratic integral
of~\eqref{EqCylindrical} in the form:
$$
F = a \, p_R^2+d \, p_Rp_S+b \, p_S^2
  + f \, p_R + g \, p_S+h,
$$
where all the coefficients are unknown functions which depend on $r$, $\theta$, $z$.
Writing down the condition~\eqref{EqForIntegral}
for such an integral $F$, we obtain the system of
PDEs which splits into two parts: the first one
contains relations between the unknown functions $a$, $b$, $d$, $h$ only,
the second one is between $f$ and $g$. As in
Proposition~\ref{PropFirstOrder},
it is easy to check that if $F$ is the first
integral, then both functions $f$ and $g$ must
vanish identically.
So we start our analysis with an integral of the
form
\[
F = a(r,\theta,z) p_R^2 + d(r,\theta,z) p_R p_S
  + b(r,\theta,z) p_S^2 + h(r,\theta,z).
\]
The condition~\eqref{EqForIntegral} implies:
\begin{gather}
a_r=0,
\label{Sys1}
\\
2a_\theta+r^2(a_z+2d_r)=0,
\label{Sys2}
\\
4a+r^3d_z+2r(d_\theta+r^2b_r)=0,
\label{Sys3}
\\
2d+r^3b_z+2rb_\theta=0,
\label{Sys4}
\\
2r^3(r^4+16z^2)^{3/2}h_r-8kr^6a-16kr^5zd=0,
\label{Sys5}
\\
r(r^4+16z^2)^{3/2}(r^2h_z+2h_\theta)
-4kr^6d-32kr^5zb=0.
\label{Sys6}
\end{gather}
Integrating the equations
\eqref{Sys1}--\eqref{Sys3} successively,
we obtain
\begin{gather*}
  a(r,\theta,z) = \alpha(\theta,z),
  \qquad
  d(r,\theta,z) = \gamma(\theta,z)
  - \frac{1}{2} r \alpha_z
  + \frac{\alpha_\theta}{r},
\\
  b(r,\theta,z) = \frac{\alpha}{r^2}
  + \frac{\alpha_{\theta\theta}}{2r^2}
  + \frac{r^2 \alpha_{zz}}{8}
  + \frac{\gamma_\theta}{r}
  - \frac{r\gamma_z}{2}
  + \omega(\theta,z),
\end{gather*}
where $\alpha(\theta,z)$, $\gamma(\theta,z)$,
$\omega(\theta,z)$ are arbitrary functions.
Then~\eqref{Sys4} takes the form
\[
  \alpha_{zzz}r^6 - 4\gamma_{zz}r^5
  + (2\alpha_{\theta zz} + 8\omega_z)r^4
  + 4(\alpha_{\theta \theta z} + 4\omega_\theta)r^2
  + 16(\gamma_{\theta\theta} + \gamma)r
  + 8(\alpha_{\theta\theta\theta} + \alpha_\theta) = 0.
\]
This is a polynomial in $r$ with coefficients
depending on $\theta$, $z$ only. Since this
polynomial must vanish, all its coefficients must
vanish as well. This allows one to find
$\alpha(\theta,z)$, $\gamma(\theta,z)$,
$\omega(\theta,z)$ and, consequently, the
coefficients $a(r,\theta,z)$, $b(r,\theta,z)$,
$d(r,\theta,z)$ explicitly. We omit these long but
simple calculations and skip the final form of
these coefficients since they are quite cumbersome.

After that we are left with two equations
\eqref{Sys5}, \eqref{Sys6} on the unknown function $h(r,\theta,z)$ which take the form:
\begin{gather}
  (r^4+16z^2)^{3/2}h_r
  + 2kr(c_8 r^2 - z(4 c_9 + z(4 c_3+c_5 r^2+4c_6z)))\cos (2\theta)
  \notag
\\
  -2kr(c_9r^2+z(4c_8+z(4c_2-c_6r^2+4c_5z)))
  \sin (2\theta)
  - 4kr^3(c_7-c_4z^2)
  \notag
\\
  -8kr^2z ((s_2+s_4z)\cos \theta
  +(s_3+s_5z)\sin \theta) = 0,
\label{Sys7}
\end{gather}
\begin{gather}
  (r^4+16z^2)^{3/2}(r^2h_z+2h_\theta)+2c_1kr^2(r^4-16z^2)
\notag
\\
-kr^2(4c_9r^2+c_2(r^4+16z^2)-2z(-8c_8+2c_3r^2+c_5r^4+6c_6r^2z-8c_5z^2))\cos (2\theta)
\notag
\\
+kr^2(-4c_8r^2+c_3(r^4+16z^2)+2z(8c_9+2c_2r^2-c_6r^4+6c_5r^2z+8c_6z^2))\sin (2\theta)
\notag
\\
-4kr^3(s_2r^2+z(8s_3-3s_4r^2+8s_5z))\cos \theta
+4kr^3(-s_3r^2+z(8s_2+3s_5r^2+8s_4z))\sin \theta
\notag
\\
-4kr^2z(8c_7+8s_1r^2+c_4(r^4+8z^2))=0.
\label{Sys8}
\end{gather}
Here $c_k$, $s_k$ are arbitrary constants.
The equation~\eqref{Sys7} can be integrated.
However,  the general solution $h(r,\theta,z)$
to~\eqref{Sys7} is expressed in terms of elliptic
integrals. We consider the simplest case
$$
s_2=s_3=s_4=s_5=0.
$$
In this case $h(x,y,z)$ can be found from~\eqref{Sys7}
in terms of elementary functions as follows:
$$
h(r,\theta,z) = \psi(\theta,z) + \frac{k}{4z \sqrt{r^4+16z^2}}\phi(r,\theta,z),
$$
where $\psi(\theta,z)$ is an arbitrary function and
\begin{multline*}
\phi(r,\theta,z) = (c_9r^2+z(4c_8+c_3r^2+c_6r^2z-4c_5z^2))\cos (2\theta)
\\
+ (c_8r^2+z(-4c_9+c_2r^2+c_5r^2z+4c_6z^2))\sin (2\theta) + 8z (c_4z^2-c_7).
\end{multline*}
The unknown function $\psi(\theta,z)$ should be
chosen such that the relation~\eqref{Sys8}
holds identically. It seems that the only possible way
to satisfy this requirement is to put
$$
\psi(\theta,z) \equiv 0, \qquad c_1=c_5=c_6=c_8=c_9=s_1=0.
$$
In this case~\eqref{Sys8} is satisfied. Thus we found all the coefficients of $F.$ Notice that $\widetilde{F}=4F-8c_7H$ is also the first integral of~\eqref{EqCylindrical} having the simpler form:
$$
\widetilde{F}
  = \widetilde{a}\,p_R^2 + \widetilde{d}\,p_Rp_S
  + \widetilde{b}\,p_S^2 + \widetilde{h},
$$
where
\begin{align*}
\widetilde{a} &= 2z(2c_4z-c_2\cos (2\theta)
  +c_3\sin (2\theta)),
\\
\widetilde{d} &
  =\Big( c_2r+\frac{4c_3z}{r} \Big) \cos(2\theta)
  - \Big( c_3x-\frac{4c_2z}{r} \Big) \sin(2\theta) -4c_4rz,
\\
\widetilde{b} &= \frac{1}{r^2} \big(
  (-c_3r^2+2c_2z) \cos(2\theta)
  -(c_2r^2+2c_3z) \sin(2\theta)
  +c_4 (r^4+4z^2) \big),
\\
\widetilde{h} &= \frac{k}{\sqrt{r^4+16z^2}}
  \big( r^2(c_3\cos (2\theta)+c_2 \sin (2\theta))+8c_4z^2 \big).
\end{align*}
Here $c_2$, $c_3$, $c_4$ are arbitrary constants.
It is left to notice that $\widetilde{F}$ is
linear in these constants, i.\,e. it has the form
$\widetilde{F}=c_2F_1+c_3F_2+c_4F_3$.
This implies that the functions
\begin{align*}
  F_1 &=
  \phantom{-} \Big( p_Rp_Sr-2p_R^2z+\frac{2p_S^2z}{r^2} \Big) \cos (2\theta)
  + \Big( -p_S^2+\frac{4 p_Rp_Sz}{r}+\frac{kr^2}{\sqrt{r^4+16z^2}} \Big) \sin (2\theta),
\\
  F_2 &=
  - \Big( p_Rp_Sr-2p_R^2z+\frac{2p_S^2z}{r^2} \Big) \sin (2\theta)
  + \Big( -p_S^2+\frac{4 p_Rp_Sz}{r}+\frac{kr^2}{\sqrt{r^4+16z^2}} \Big) \cos (2\theta),
\\
  F_3 &=
  4z^2 p_R^2 - 4rz p_Rp_S
  + \frac{r^4+4z^2}{r^2}p_S^2
  + \frac{8k z^2}{\sqrt{r^4+16z^2}}.
\end{align*}
are also integrals of~\eqref{EqCylindrical}.

\noindent
Sergey Basalaev \\
Novosibirsk State University,
1 Pirogova st., 630090 Novosibirsk Russia \\
e-mail: \texttt{s.basalaev@g.nsu.ru}

\smallskip
\noindent
Sergei Agapov \\
Sobolev Institute of Mathematics,
4 Acad. Koptyug avenue, 630090 Novosibirsk Russia \\
e-mail: \texttt{agapov@math.nsc.ru}, \texttt{agapov.sergey.v@gmail.com}

\end{document}